\documentclass[11pt, leqno]{scrartcl}
\usepackage{amsmath}
\usepackage{amssymb}
\usepackage{amsthm}
\usepackage{mathrsfs}
\usepackage[a4paper, margin=3cm]{geometry}
\usepackage{color}
\usepackage{enumerate}
\usepackage{numprint}
\usepackage{booktabs}
\usepackage{cprotect}

\usepackage[colorlinks]{hyperref}

\newtheorem{thm}{Theorem}

\newtheorem{lem}[thm]{Lemma}
\newtheorem{prop}[thm]{Proposition}
\theoremstyle{definition}
\newtheorem{ex}[thm]{Example}
\newtheorem{rem}[thm]{Remark}

\newcommand{\CC}{\mathbb{C}}
\newcommand{\diag}{\operatorname{diag}}
\newcommand{\eps}{\varepsilon}
\renewcommand{\Im}{\operatorname{Im}}

\newcommand{\NN}{\mathbb{N}}
\newcommand{\RR}{\mathbb{R}}
\newcommand{\Tr}{\operatorname{Tr}}
\newcommand{\ZZ}{\mathbb{Z}}
\newcommand{\HH}{\mathbb{H}}
\newcommand{\SSp}{\textrm{Sp}}
\newcommand{\GSp}{\textrm{GSp}}
\newcommand{\GL}{\textrm{GL}}

\newcommand{\mat}[4]{{\setlength{\arraycolsep}{0.5mm}\left[
\begin{array}{cc}#1&#2\\#3&#4\end{array}\right]}}

\title{Analytic evaluation of Hecke eigenvalues for Siegel modular forms of degree two}
\author{%
  Owen Colman\footnote{\url{owencolman@gmail.com}}, Alexandru Ghitza\footnote{\url{aghitza@alum.mit.edu}} and Nathan C. Ryan\footnote{\url{nathan.ryan@bucknell.edu}}
}
\date{\today}

\begin{document}
\thispagestyle{empty}

\maketitle
\begin{abstract}The standard approach to evaluate Hecke eigenvalues of a Siegel modular eigenform $F$ is to determine a large number of Fourier coefficients of $F$ and then compute the Hecke action on those coefficients.
  We present a new method based on the numerical evaluation of $F$ at explicit points in the upper-half space and of its image under the Hecke operators.
  The approach is more efficient than the standard method and has the potential for further optimization by identifying good candidates for the points of evaluation, or finding ways of lowering the truncation bound.
  A limitation of the algorithm is that it returns floating point numbers for the eigenvalues; however, the working precision can be adjusted at will to yield as close an approximation as needed.
\end{abstract}

\section{Introduction}\label{sect:intro}

The explicit computation of classical modular forms and their associated
L-functions has been very useful to formulate and verify conjectures, to
discover new phenomena and to prove theorems.   There are a variety of
ways to effectively compute the Fourier coefficients of classical modular forms
and, therefore, their L-functions.  Analogous work for Siegel modular
forms of degree two is less well-developed for, perhaps, two main reasons:
\begin{enumerate}
\item the methods for computing Siegel modular forms are \textit{ad
    hoc} and less efficient than those for computing classical modular
  forms;
\item computing Siegel modular forms does not immediately give you the
 associated L-functions since the Hecke eigenvalues of Siegel modular
 forms, unlike in the classical case, are not equal to the Fourier
 coefficients and because the Euler factors of the L-function involve
 knowing both the $p$th and the $p^2$th eigenvalues.
\end{enumerate}

To give an idea of the difficulty of computing the L-function of a
Siegel modular form, we consider an example.  Let $\Upsilon_{20}$ be
the unique normalized Siegel modular form of degree 2 and weight 20 that is a
Hecke eigenform and not a Saito-Kurokawa lift.  Skoruppa
\cite{skoruppa} gave an explicit formula for $\Upsilon_{20}$ in terms of the generators of the ring of Siegel modular forms of degree 2 and the
largest calculation of $\Upsilon_{20}$ has been carried out by Kohnen and Kuss
\cite{kohnen-kuss} (we point out that Kurokawa \cite{kurokawa1, kurokawa2} was the first to compute $\Upsilon_{20}$ but his computations were not very extensive).  The computation that Kohnen and Kuss carried out
was enough to find the $p$th eigenvalue for $p\leq 997$ and the
$p^2$th eigenvalue for $p\leq 79$.  They compute Fourier coefficients
indexed by quadratic forms with discriminant up to 3000000 and then use
them to determine the Hecke eigenvalues.  An examination of the formulas on page
387 of \cite{skoruppa} shows that to find the eigenvalue $\lambda(n)$
of $T_n$, for $n = p^2$, requires the Fourier coefficients indexed by
quadratic forms of discriminant up to $n^2 = p^4$.  This relation makes
it infeasible to compute many more Fourier coefficients, and thus
Hecke eigenvalues, using this approach.  Instead, in this paper, we propose a different approach.

Our method does not compute \emph{any} of the Fourier coefficients of the Siegel modular form being studied.  Instead, we take suitable truncations of the Fourier expansions of the \emph{Igusa generators} (whose coefficients are inexpensive to compute) and use these truncations to evaluate our modular form numerically at points in the upper half space. This approach is based on work of Br\"oker and
Lauter \cite{broker} in which they use such techniques to evaluate Igusa functions.  Using their
method we find the eigenvalue $\lambda(p)$ of an eigenform $F$ by doing the
following:
\begin{itemize}
\item evaluate $F$ at some point $Z$ in the Siegel upper
half-space;
\item evaluate $F|T_p$ at the same point $Z$;
\item take the ratio $(F|T_p)(Z)/F(Z)$.
\end{itemize}
The conceptual shift that we are proposing is that, instead of
representing the Siegel modular form $F$ as a list of Fourier
coefficients, we represent $F$ by its values at points in the Siegel
upper half-space.  The idea is simple but its importance
can be seen by virtue of the results.  We remark that in \cite{analytic1} we describe an implementation of the analogous method for classical modular forms and, in some cases, outperform the standard method using modular symbols.

The potential to parallelize our algorithm stems from the fact that we sum over the coset decomposition of the Hecke operators, and the computation of each summand is independent; these computations can therefore be performed in parallel.
Such approaches have been used in the past, for instance in determining the Hecke eigenvalues of paramodular forms, see~\cite{pooryuen}.
We thank the referees for pointing this out, and note that the similarity ends at the level of the sum itself: Poor and Yuen specialize the paramodular eigenform to a modular curve, then compute the summands (which are power series in one variable) exactly.
We work with the Siegel eigenform itself (as a power series in three variables) and compute good numerical approximations to the summands.

It is important to emphasize that our method takes as input the expression of a Siegel eigenform as a polynomial in the Igusa generators.
Our objective is then to efficiently compute approximate values of the Hecke eigenvalues.
We do not claim to obtain further information about the Fourier coefficients of the eigenform, nor that this is an efficient way of determining the exact value of the eigenvalues (unless the latter happen to be integers).

The paper is organized as follows.
We begin by stating some numerical preliminaries used in our method.
Then, we give the relevant background on Siegel modular forms and discuss Br\"oker and Lauter's work and how to compute $F|T_p$ both in theory and in practice.
We conclude by presenting some results of our computations, together with details of the implementation and ideas for further improvement.

\subsection*{Acknowledgments:}  We thank John Voight for proposing this
project to us.  We also thank the anonymous referees for many helpful suggestions.


\section{Numerical preliminaries}

Before we describe our algorithm to compute Hecke eigenvalues of Siegel modular forms analytically, we begin by stating some results related to bounding the error introduced when we evaluate a given Siegel modular form and its image under the Hecke operators $T_p$ and $T_{p^2}$ at a point in the upper half-plane.

\subsection{Error in quotient}
We have a quantity defined as
\begin{equation*}
  z=\frac{x}{y}\quad\text{with }x,y\in\CC.
\end{equation*}
The numerator and denominator can be approximated to $x_A$, resp. $y_A$; we define $z_A:= \frac{x_A}{y_A}$.
Given $\eps>0$, what values of $\eps_x$ and $\eps_y$ ensure that
\begin{equation*}
  \text{if }|x-x_A|<\eps_x\text{ and }|y-y_A|<\eps_y\text{ then }|z-z_A|<\eps?
\end{equation*}

\begin{lem}
  With the above notation, let $e_x=x-x_A$ and $e_y=y-y_A$.
  Then
  \begin{equation*}
    z-z_A=\frac{e_x-e_yz_A}{y_A+e_y}.
  \end{equation*}
\end{lem}
\begin{proof}
  Straightforward calculation.
\end{proof}

\begin{prop}
  \label{prop:quot}
  For any $h\in (0,1)$, if
  \begin{equation*}
    \eps_x<\frac{h\eps|y_A|}{2}
    \quad\text{and}\quad
    \eps_y<\min\left\{\frac{(1-h)\eps |y_A|}{2|z_A|}, \frac{|y_A|}{2}\right\},
  \end{equation*}
  then $|z-z_A|<\eps$.
\end{prop}
\begin{proof}
  Under the hypotheses, we have $|y_A+e_y|>|y_A|/2$ so
  \begin{equation*}
    |z-z_A|<\frac{2}{|y_A|}\left(|e_x|+|e_yz_A|\right)
    <h\eps+(1-h)\eps=\eps.
  \end{equation*}
\end{proof}

The value of the parameter $h$ can be chosen in such a way that the calculations of $x_A$ and of $y_A$ are roughly of the same level of difficulty.

In order to use the results of Proposition~\ref{prop:quot} in practice, we need a lower bound on $|y_A|$ and an upper bound on $|z_A|$ (which can be obtained from the lower bound on $|y_A|$ and an upper bound on $|x_A|$).
How do we bound $|x_A|$?
We compute a very coarse estimate $\tilde{x}$ to $x$, with $\tilde{\eps}_x$ just small enough that $|\tilde{x}|-2\tilde{\eps}_x>0$.
(We can start with $\tilde{\eps}_x=0.1$ and keep dividing by $10$ until the condition holds.)
Later we will make sure that $\eps_x$ is smaller than $\tilde{\eps}_x$.
Then we know that
\begin{equation*}
  |\tilde{x}-x|<\tilde{\eps}_x\qquad\text{and}\qquad
  |x_A-x|<\eps_x\leq\tilde{\eps}_x,
\end{equation*}
so
\begin{equation*}
\big||x_A|-|\tilde{x}|\big|\leq |x_A-\tilde{x}|<2\tilde{\eps}_x\qquad\Rightarrow\qquad 0< |\tilde{x}|-2\tilde{\eps}_x < |x_A| < |\tilde{x}|+2\tilde{\eps}_x,
\end{equation*}
giving us lower and upper bounds on $|x_A|$.
A similar argument works for $|y_A|$.

\section{Siegel modular forms}\label{sec:smf}

Let the symplectic group of similitudes of genus $2$ be defined by
\begin{multline*}
 \GSp(4) := \{G \in \GL(4) : {}^{t}G J G = \lambda(G) J,
 \lambda(G) \in \GL(1) \} \\ \mbox{where } J = \mat{}{I_2}{-I_2}{}.
\end{multline*}
Let $\SSp(4)$ be the subgroup with $\lambda(G)=1$. The group $\GSp^+(4,\RR) := \{ G \in\GSp(4,\RR) : \lambda(G) > 0 \}$ acts on the Siegel upper half space $\HH_2 := \{ Z \in M_2(\CC) : {}^{t}Z = Z, \Im(Z) > 0\}$ by
\begin{equation}
 G \langle Z \rangle := (AZ+B)(CZ+D)^{-1}, \quad \text{where } G = \mat{A}{B}{C}{D} \in \GSp^+(4,\RR),  Z \in \HH_2.
\end{equation}

Let $S_k^{(2)}$ be the space of
holomorphic Siegel cusp forms of weight $k$, genus $2$ with respect to $\Gamma^{(2)} := \SSp(4,\ZZ)$. Then $F \in S_k^{(2)}$
satisfies
$$
 F(\gamma \langle Z \rangle) = \det(CZ+D)^{k} F(Z)
$$
for all $\gamma = \mat{A}{B}{C}{D} \in \Gamma^{(2)}$ and  $Z \in
\HH_2$.  This also can be written in terms of the slash operator: for
$M\in\GSp^+(4,\RR)$ let $\left(F\vert_k M\right)(Z) = \det(CZ+D)^{-k}F(M\langle
Z\rangle )$.  Then the
functional equation satisfied by a Siegel modular form can be written
as:
\[
\left(F\vert_k M\right)(Z)=F(z)
\]
for all $M\in \SSp(4,\ZZ)$.

Now we describe the Hecke operators acting on $S_k^{(2)}$.
For $M \in \GSp^+(4,\RR) \cap M_{4}(\ZZ)$, define the Hecke operator
$T(\Gamma^{(2)} M \Gamma^{(2)})$ on $S_k^{(2)}$ as in \cite[(1.3.3)]{andrianov}.  For a positive integer $m$,
we define the Hecke operator $T_m$ by
\begin{equation}\label{hecke-op-m-defn}
T_m := \sum\limits_{\lambda(M)=m} T(\Gamma^{(2)} M
\Gamma^{(2)}).
\end{equation}
See Section~\ref{sec:hecke} for an explicit decomposition of the
double cosets $T_p$ and $T_{p^2}$ into right cosets.  Suppose
\[
T_m = \sum \Gamma^{(2)}\alpha
\]
is a right coset decomposition of the Hecke operator $T_m$.  Then
the operator $T_m$ acts on a Siegel modular form $F$ of weight $k$
as
\[
\left( F\vert_k T_m\right)(Z) = \sum \left( F\vert_k
  \alpha\right)(Z).
\]
This action can be described in terms of the Fourier coefficients of
the Siegel modular form $F$.

Any Siegel modular form $F$ of degree $2$ has a Fourier expansion of the form
\begin{equation*}
  F(Z) = \sum_N a_N(F) \exp\left(2\pi i\Tr(NZ)\right)\qquad a_N(F)\in\CC,
\end{equation*}
where the sum ranges over all positive semi-definite matrices
$N=\begin{pmatrix}a&b/2\\b/2&c\end{pmatrix}$ with $a,b,c\in\ZZ$.  The
quadratic form $N$ is often written $[a,b,c]$ using Gauss's notation.
Using the decompositions of the Hecke operators in
Section~\ref{sec:hecke} one can derive formulas for the
action of $T_p$ and $T_{p^2}$ on a Siegel modular form $F$.  When
these formulas are written down as in \cite[p. 387]{skoruppa} one can
see that to compute $\lambda_F(p)$, the Hecke eigenvalue of $F$ with
respect to the Hecke operator $T_p$, one needs Fourier coefficients
up to discriminant of order $p^2$.  To compute $\lambda_F(p^2)$, the
Hecke eigenvalue of $F$ with respect to the Hecke operator $T_{p^2}$,
one needs Fourier coefficients up to discriminant $p^4$.  With
current methods, computing this number of coefficients of a Hecke
eigenform that is not a Saito-Kurokawa lift has proven impossible.

A bottleneck to computing such a large number of coefficients is the
fact that there is no known way to compute individual coefficients in
parallel.  The determination of a single Fourier coefficient requires
knowledge of many other Fourier coefficients.  Our method, described
above, has approximately the same number of steps to compute a new
Hecke eigenvalue but these steps, in our method, are easily done in parallel.

\section{Evaluating Hecke eigenforms}\label{sec:evaluating}

\subsection{Bounds on the coefficients of the Igusa generators}

\begin{prop}
  \label{prop:bounds}
  Let $E_4$, $E_6$, $\chi_{10}$ and $\chi_{12}$ denote the Igusa generators of
  the ring of even-weight Siegel modular forms of genus $2$ with respect
  to $\SSp(4,\ZZ)$.

  We have the following bounds on the Fourier coefficients of these forms:
  \begin{align*}
    \left|a_N(E_4)\right| &< \numprint{19230}\,t^5,\\
    \left|a_N(E_6)\right| &< \numprint{12169}\,t^9,\\
    \left|a_N(\chi_{10})\right| &< \frac{1}{236}\,A(\eps,9)\,t^{9+\eps},\\
    \left|a_N(\chi_{12})\right| &< \frac{1}{311}\,A(\eps,11)\,t^{11+\eps},
  \end{align*}
  where the last two hold for any $\eps>0$, $t=\Tr(N)$, and the function $A(\eps,s)$ is defined by
  \begin{equation*}
    A(\eps,s)=\frac{1}{(2\pi)^{1/4}}\,\exp\left(9\eps^{-1}2^{3/\eps}/\log(2)\right)\,\zeta(1+\eps)
    \,\max\left\{1,\sqrt{\frac{\Gamma(s+1/2+\eps)}{\Gamma(s-1/2-\eps)}}\right\}.
  \end{equation*}
\end{prop}
\begin{proof}
  It follows directly from~\cite[Corollary 3.6 and Remark 3.7]{broker} that
  \begin{align*}
      \left|a_N(E_4)\right| &< \numprint{19230}\left(4ac-b^2\right)^{5/2}\leq\numprint{19230}\Tr(N)^5,\\
      \left|a_n(E_6)\right| &< \numprint{12169}\left(4ac-b^2\right)^{9/2}\leq\numprint{12169}\Tr(N)^9.
  \end{align*}
The second two inequalities follow from~\cite[Theorem 5.10]{broker} with $\gamma=\eta=\eps/3$.  
\end{proof}

\begin{rem}
  The bounds for $\chi_{10}$ and $\chi_{12}$ in Proposition~\ref{prop:bounds}
  allow for further optimization by choosing the parameter $\eps$ appropriately.

  Considering $\chi_{10}$, the factor $t^{9+\eps}$ is of course dominant as
  $t\to\infty$, but choosing $\eps$ as small as possible is counterproductive
  for practical computations, as the factor $A(\eps, 9)$ explodes for small $\eps$.

  In our computations, we use $\eps=2$, so the bounds can be summarized as:
  \begin{align*}
    \left|a_N(E_4)\right| &< \numprint{19230}\,t^5,\\
    \left|a_N(E_6)\right| &< \numprint{12169}\,t^9,\\
    \left|a_N(\chi_{10})\right| &< \numprint{220439}\,t^{11},\\
    \left|a_N(\chi_{12})\right| &< \numprint{287248}\,t^{13},\\
  \end{align*}
  where $t=\Tr(N)$.
\end{rem}

\subsection{The truncation error for Siegel modular forms}
\label{sect:trunc_error}

Let $F$ be a Siegel modular form of degree $2$, with Fourier expansion
\begin{equation*}
  F(Z) = \sum_N a_N(F) \exp\left(2\pi i\Tr(NZ)\right).
\end{equation*}
Given a positive integer $T$, we will truncate the Fourier expansion of $F$ by considering only those indices $N$ whose trace is at most $T$:
\begin{equation*}
  F_T(Z) = \sum_{\Tr(N)\leq T} a_N(F) \exp\left(2\pi i\Tr(NZ)\right).
\end{equation*}

\begin{lem}
  For any $t\in\NN$, the number of Fourier indices of trace $t$ satisfies
  \begin{equation*}
    \#\{N\mid\Tr(N)=t\}\leq (t+1)(2t+1)=2t^2+3t+1\leq 6t^2.
  \end{equation*}
\end{lem}
\begin{proof}
  We have
  \begin{equation*}
    \#\{N\mid\Tr(N)=t\}=\sum_{a=0}^t\left(1+2\left\lfloor 2\sqrt{a(t-a)}\right\rfloor\right).
  \end{equation*}
  There are $t+1$ terms in the sum, and the largest corresponds to $a=t/2$ (or $a=(t-1)/2$ if $t$ is odd).
  In any case, every term in the sum is at most $1+2t$.
\end{proof}

Suppose we have, like in Proposition~\ref{prop:bounds}, an upper bound on the Fourier coefficients of $F$:
\begin{equation}
  \label{eq:anf_bound}
  |a_N(F)|\leq Ct^d\qquad\text{where $C\in\RR_{>0}$, $d\in\NN$ and $t=\Tr(N)$}.
\end{equation}
We are interested in bounding the gap between the true value $F(Z)$ and its approximation $F_T(Z)$.

\begin{prop}
  Suppose $F$ is a Siegel modular form of degree two whose Fourier coefficients
  satisfy Equation~\eqref{eq:anf_bound}, $Z\in\HH_2$ and we wish to approximate
  the value $F(Z)$ with error at most $10^{-h}$.
  It is then sufficient to use the truncation $F_T(Z)$ containing all terms of
  the Fourier expansion of $F$ with indices of trace at most $T$, where
  \begin{equation*}
    T>\frac{d+2}{\alpha(Z)}\qquad\text{and}\qquad
      6C\frac{d+3}{\alpha(Z)}\exp(-\alpha(Z) T)T^{d+2}<10^{-h}.
  \end{equation*}
  Here
  \begin{equation*}
    \delta(Z)=\sup\left\{\delta^\prime\in\RR\mid \Im(Z)-\delta^\prime I\text{ is positive semi-definite}\right\}
  \end{equation*}
  and $\alpha(Z)=2\pi\delta(Z)$.
\end{prop}

\begin{proof}
  Using~\cite[Lemma~6.1]{broker}, we have
  \begin{align*}
    \left|F(Z)-F_T(Z)\right| &= \left|\sum_{\Tr(N)>T} a_N(F) \exp\left(2\pi i\Tr(NZ)\right)\right|\\
                             &\leq \sum_{\Tr(N)>T} \left|a_N(F)\right| \left|\exp\left(2\pi i\Tr(NZ)\right)\right|\\
                               &\leq \sum_{\Tr(N)>T} \left|a_N(F)\right| \exp\left(-\alpha(Z)\Tr(N)\right)\\
                               &<\sum_{t=T+1}^\infty \sum_{\Tr(N)=t} \left|a_N(F)\right|\exp\left(-\alpha(Z)t\right)\\
                               &\leq\sum_{t=T+1}^\infty 6Ct^{d+2}\exp\left(-\alpha(Z)t\right)\\
                               &\leq 6C\int_T^\infty x^{d+2}\exp\left(-\alpha(Z)x\right)\,dx\\
                               &= 6C\exp(-\alpha(Z) T)\sum_{j=0}^{d+2}\frac{(d+2)!}{j!\alpha(Z)^{d-j+3}}T^j\\
                               &<\frac{6C(d+3)}{\alpha(Z)}\exp(-\alpha(Z) T)T^{d+2},
  \end{align*}
  where the last inequality holds if $T$ is in the half-infinite interval on
  which the integrand is decreasing (i.e.\ $T>(d+2)/\alpha(Z)$).
\end{proof}

\begin{ex}
  We determine $T$ sufficient for computing $E_4(Z)$ within $10^{-20}$ at the point
  \begin{equation*}
    z=\begin{pmatrix}
      5i & i\\
      i & 6i
     \end{pmatrix}.
  \end{equation*}
  We have
  \begin{equation*}
      \alpha(Z)=27.5327
  \end{equation*}
  so we are looking for $T$ such that
  \begin{equation*}
      \exp(-\alpha(Z)T)T^7<2.983\cdot 10^{-25},
  \end{equation*}
  which is easily seen (numerically) to hold as soon as $T\geq 3$.

  We proceed similarly to obtain the values in Table~\ref{table:truncation}.

  \begin{table}[h]
    \centering
    \begin{tabular}{lrrrr}
      \toprule
      & \multicolumn{4}{c}{$T$}\\
      \cmidrule(r){2-5}
      error & $E_4$ & $E_6$ & $\chi_{10}$ & $\chi_{12}$ \\
      \midrule
      $10^{-10}$   & $2$  & $2$  & $2$  & $2$\\
      $10^{-20}$   & $3$  & $3$  & $3$  & $3$\\
      $10^{-100}$  & $10$ & $10$ & $10$ & $11$\\
      $10^{-1000}$ & $86$ & $86$ & $87$ & $87$\\
      \bottomrule
    \end{tabular}
    \caption{Truncation necessary for computing $F(Z)$ within specified error}
    \label{table:truncation}
  \end{table}
\end{ex}



\section{Our method}

As described above, our method is rather straightforward.  We fix a
$Z\in \HH^2$ and evaluate $F(Z)$, using methods in
Section~\ref{sec:evaluating}.  Consider the double coset $T_p=\sum \Gamma^{(2)}\alpha$ and its action on $F$:
\[
\left( F\vert_k T_p \right)(Z)=\sum\left( F\vert_k\alpha\right)(Z).
\]
What is left to do, then, is, for $\alpha$ in the decomposition, to
compute $\left(F\vert_k \alpha\right)(Z)$.  That is, to be able to
write $\alpha$ as $\mat{A}{B}{C}{D}$ and to be able to evaluate
\[
\det(CZ+D)^{-k}F\left((AZ+B)(CZ+D)^{-1}\right).
\]
In Section~\ref{sec:hecke} we present the desired decompositions for
the Hecke operators $T_p$ and $T_{p^2}$ and we use the methods of
Section~\ref{sec:evaluating} to evaluate the Siegel modular form at
the points $(AZ+B)(CZ+D)^{-1}\in\HH^2$.

\subsection{Hecke action}\label{sec:hecke}

Hecke operators are defined in terms of double cosets $\Gamma M \Gamma$
and the action of such an operator is determined by the right cosets
that appear in the decomposition of these double cosets.  For a prime $p$ we consider the double coset
$T_p=\Gamma^{(2)} \, \diag(1,1,p,p)\Gamma^{(2)}$.  An explicit version
of a formula, due to Andrianov, for the right cosets that appear in the decomposition of $T_p$, is given by Cl\'ery and van der Geer:
\begin{prop}\cite{andrianov,Clery14}\label{prop:tp}
The double coset $T_p$ admits the following left coset decomposition:
  \begin{multline*}
    \Gamma^{(2)}
\left(
\begin{smallmatrix}
p & 0 & 0 & 0 \\
0 & p & 0 & 0 \\
0 & 0 & 1 & 0 \\
0 & 0 & 0 & 1
\end{smallmatrix}
\right)
+
\sum_{0 \leq a,b,c \leq p-1}
\Gamma^{(2)}
\left(
\begin{smallmatrix}
1 & 0 & a & b \\
0 & 1 & b & c \\
0 & 0 & p & 0 \\
0 & 0 & 0 & p
\end{smallmatrix}
\right)
+\\
\sum_{0 \leq a \leq p-1}
\Gamma^{(2)}
\left(
\begin{smallmatrix}
0 & -p & 0 & 0 \\
1 & 0 & a & 0 \\
0 & 0 & 0 & -1 \\
0 & 0 & p & 0
\end{smallmatrix}
\right)
+
\sum_{0 \leq a,m \leq p-1}
\Gamma^{(2)}
\left(
\begin{smallmatrix}
p & 0 & 0 & 0 \\
-m & 1 & 0 & a \\
0 & 0 & 1 & m \\
0 & 0 & 0 & p
\end{smallmatrix}
\right)
\end{multline*}
and we have that the degree of $T_p$ is $p^3+p^2+p+1$.
\end{prop}
Thus, in particular, in order to find $\lambda_p$, then, $p^3+p^2+p+1$ independent
evaluations of our Siegel modular form $F$ at points in $\HH^2$ are
required.  This is why our method is so amenable to parallelization.

Similarly, for a prime $p$ define the operator $T_{p^2}$ as a sum of double cosets:
$$
T_{p^2}=
\Gamma^{(2)}
 \left(
\begin{smallmatrix}
p & 0 & 0 & 0 \\
0 & p & 0 & 0 \\
0 & 0 & p & 0 \\
0 & 0 & 0 & p
\end{smallmatrix}
\right)
 \Gamma^{(2)}
 +
\Gamma^{(2)}
 \left(
\begin{smallmatrix}
1 & 0 & 0 & 0 \\
0 & p & 0 & 0 \\
0 & 0 & p^2 & 0 \\
0 & 0 & 0 & p
\end{smallmatrix}
\right)
 \Gamma^{(2)}
 +
\Gamma^{(2)}
 \left(
\begin{smallmatrix}
1 & 0 & 0 & 0 \\
0 & 1 & 0 & 0 \\
0 & 0 & p^2 & 0 \\
0 & 0 & 0 & p^2
\end{smallmatrix}
\right)
 \Gamma^{(2)}\\
$$

Again, based on a result of Andrianov, Cl\'ery and van der Geer give an
explicit decomposition of the operator $T_{p^2}$:
\begin{prop}\cite{andrianov,Clery14}\label{prop:tp2}
The Hecke operator $T_{p^2}$
has degree $p^6+p^5+2p^4+2p^3+p^2+p+1$ and
admits a known explicit left coset decomposition.
\end{prop}

One can do better, however; we can reduce the number of summands at which we need to evaluate $F$ to be $\mathcal{O}(p^4)$ instead of $\mathcal{O}(p^6)$ by using some standard facts about the Hecke algebra for Siegel modular forms of degree 2.  The Hecke operator $T_{p^2}$ is itself a linear combination of three double cosets:
\begin{multline}\label{eq:gens}
T_{p^2,0} = \Gamma^{(2)} \diag (p,p;p,p)\Gamma^{(2)},\, T_{p^2,1} = \Gamma^{(2)} \diag (1,p;p^2,p)\Gamma^{(2)},\, \text{ and }\\ T_{p^2,2} = \Gamma^{(2)} \diag (1,1;p^2,p^2)\Gamma^{(2)}.
\end{multline}
The decomposition in Proposition~\ref{prop:tp2} is itself the (disjoint) sum of the decomposition of three double cosets $T_{p^2,0}$, $T_{p^2,1}$ and $T_{p^2,2}$.  

The $p$-part of the Hecke algebra is generated by the operators $T_p$, $T_{p^2,0}$ and $T_{p^2,1}$ and, in fact, in \cite{krieg,vdG}, it is shown that
\begin{equation}\label{eq:tp-relation}
(T_p)^2 = T_{p^2,0} + (p+1)T_{p^2,1} + (p^2+1)(p+1) T_{p^2,2}.
\end{equation}

To determine the eigenvalue $\lambda_{p^2}(F)$ for $F\in S_k^{(2)}$ with respect to the Hecke operator $T_{p^2}$, using Proposition~\ref{prop:tp}, we first find the eigenvalue $\lambda_p(F)$ for the operator $T_p$.  Then, we find the eigenvalues $\lambda_{p^2,0}(F)$ (known to be $p^{-2k}$ by the definitions in Section~\ref{sec:smf}) and the eigenvalue $\lambda_{p^2,1}(F)$ for the operator $T_{p^2,1}$.  Then using \eqref{eq:tp-relation} we can find the eigenvalue $\lambda_{p^2,2}(F)$ for the operator $T_{p^2,2}$.  Putting it all together, then, all we need is an explicit decomposition of $T_{p^2,1}$ into left cosets, in order to compute $\lambda_{p^2}(F)$.

\begin{prop}[\cite{andrianov}]\label{prop:tp2-cosets}
The Hecke operator $T_{p^2,1}$ admits the following left coset decomposition:
\begin{multline*}
\sum_{0\leq \alpha < p}\Gamma^{(2)} \left(
\begin{smallmatrix}
p^2 & 0 & 0 & 0\\
-p\alpha & p & 0 & 0\\
0 & 0 & 1 & \alpha\\
0 & 0 & 0 & p\\
\end{smallmatrix}
\right) \Gamma^{(2)} +
\Gamma^{(2)} \left(
\begin{smallmatrix}
p & 0 & 0 & 0\\
0 & p^2 & 0 & 0\\
0 & 0 & p & 0\\
0 & 0 & 0 & 1\\
\end{smallmatrix}
\right) \Gamma^{(2)} 
+
\sum_{\substack{0\leq a,b,c <p\\ ac-b^2\equiv 0\pmod{p}\\\text{ and not all zero}}}
\Gamma^{(2)} \left(
\begin{smallmatrix}
p & 0 & a & b\\
0 & p & b & c\\
0 & 0 & p & 0\\
0 & 0 & 0 & p\\
\end{smallmatrix}
\right)\Gamma^{(2)} 
+\\
\sum_{\substack{0\leq \alpha,\beta<p\\0\leq C <p^2}}
\Gamma^{(2)} \left(
\begin{smallmatrix}
p & 0 & 0 & p\beta\\
-\alpha & 1 & \beta & \alpha\beta+C\\
0 & 0 & p & p\alpha\\
0 & 0 & 0 & p^2\\
\end{smallmatrix}
\right)\Gamma^{(2)}
+
\sum_{\substack{0\leq \beta <p\\0\leq A < p^2}}
\Gamma^{(2)}
 \left(
\begin{smallmatrix}
1 & 0 & A & \beta\\
0 & p & p\beta & 0\\
0 & 0 & p^2 & 0\\
0 & 0 & 0 & p\\
\end{smallmatrix}
\right)
\Gamma^{(2)}.
\end{multline*}
Thus the degree of $T_{p^2,1}$ is $p^4+p^3+p^2+p$.
\end{prop}

\begin{rem}In the Introduction, we discussed the difficulty of computing $\lambda_{p^2}(F)$ using the action of $T_{p^2}$ on the coefficients of the eigenform $F$.  One might ask whether if we could more efficiently compute $\lambda_{p^2}(F)$ using the action of $T_{p^2,1}$ on $F$ as described in Proposition~\ref{prop:tp2-cosets} and \eqref{eq:tp-relation}.  It turns out, though, that one still would require coefficients up to discriminant $p^4$ using $T_{p^2,1}$ and \eqref{eq:tp-relation}.  
\end{rem}

\section{Some computations and implementation details}

We describe some sample computations involving the eigenform of smallest weight that is not a lift from lower rank groups, namely the cusp form $\Upsilon_{20}$ mentioned in the introduction:
\begin{equation*}
  \Upsilon_{20}=
  -E_4^2\chi_{12}
  -E_4E_6\chi_{10}
  +1785600\chi_{10}^2.
\end{equation*}

As a gauge of the performance of the algorithm, we compared the timings to those required by the implementation~\cite{takemori-code} of the standard method\footnote{The only other publicly-available implementation we are aware of is~\cite{smf-code-sage}.  We did not compare against it for two reasons: (a) at the moment, the computation of the Hecke image appears to be incorrect for primes that are congruent to $1$ mod $4$ and (b) it uses Cython for the most expensive part of the computation, namely the multiplication of the $q$-expansions.  Since both our code and S. Takemori's are pure Python, we deemed this to be a more useful comparison of the two algorithms.} by Sho Takemori.

We implemented the method described in this paper in SageMath~\cite{sage}; this implementation is available at~\cite{hecke-analytic-siegel}.
The benchmarks described below were performed using a single core of a Linux machine with an i7-6700 CPU at 3.40GHz and 64GB of RAM, via the following helper functions:

\begin{verbatim}
def ups20_eigenvalue_numerical(p, prec, y11):
    CRING = _initialise_rings(prec, 2*p)
    Z = matrix(CRING, 2, 2, [y11*i, i, i, (y11+1)*i])
    R.<a, b, c, d> = QQ[]
    f = -a^2*d-a*b*c+1785600*c^2
    return _eigenvalue_T_fixed_trace(f, Z, p, 2*p)
\end{verbatim}

\begin{verbatim}
def ups20_eigenvalue_standard(p):      
    with degree2_number_of_procs(1):    
        a = eisenstein_series_degree2(4, p)
        b = eisenstein_series_degree2(6, p)
        c = x10_with_prec(p)
        d = x12_with_prec(p)    
        f = -a^2*d-a*b*c+1785600*c^2              
        return f.hecke_eigenvalue(p)    
\end{verbatim}

  \begin{table}[h]
    \centering
    \begin{tabular}{rrrrr}
      \toprule
      $p$ & $y_{11}$ & precision (bits) & numerical (s) & standard (s)\\
      \midrule
      $2$  & $2.7$ & $37$ & $0$ & $0$\\
      $3$  & $4.3$ & $62$ & $0$ & $0$\\
      $5$  & $6.1$ & $101$ & $0$ & $0$\\
      $7$  & $7.5$ & $130$ & $1$ & $1$\\
      $11$ & $9.5$ & $172$ & $3$ & $7$\\
      $13$ & $10.3$ & $190$ & $6$ & $15$\\
      $17$ & $10.9$ & $208$ & $16$ & $55$\\
      $19$ & $11.9$ & $226$ & $25$ & $90$\\
      $23$ & $12.3$ & $240$ & $54$ & $230$\\
      $29$ & $13.5$ & $267$ & $140$ & $735$\\
      $31$ & $13.9$ & $275$ & $186$ & $1185$\\
      $37$ & $14.5$ & $295$ & $406$ & $2876$\\
      \bottomrule
    \end{tabular}
    \caption{Benchmarks comparing the numerical and standard algorithms for computing the Hecke eigenvalues of $\Upsilon_{20}$. The timings are rounded to the nearest second.  The working precision was chosen so that the eigenvalue is the closest integer to the computed floating point number.}
    \label{table:ups20}
  \end{table}

For the standard algorithm, the most expensive step appears to be the multiplication of the $q$-expansions of the Igusa generators.
In the case of our numerical algorithm, the majority of the time is spent evaluating truncations of the $q$-expansions of the Igusa generators at various points in the Siegel upper half space.
These functions are polynomials in the variables $q_1$, $q_2$, $q_3$ and $q_3^{-1}$, where
\begin{equation*}
  Z=\begin{pmatrix}z_1&z_3\\z_3&z_2\end{pmatrix}
  \qquad\text{and}\qquad q_j=e^{2\pi i z_j}.
\end{equation*}
To evaluate such functions efficiently at a large number of points, we implemented an iterative version of Horner's method; to illustrate what is involved, here is how the truncation of the Igusa generator $\chi_{10}$ at trace up to $3$ is evaluated:
\begin{multline*}
  q_1\Bigg(q_2\Big(q_3^{-1}-2+q_3+q_2\left(-2q_3^{-2}-16q_3^{-1}+36-16q_3-2q_3^2\right)\Big)\\+q_1\Big(q_2\left(-2q_3^{-1}-16q_3^{-1}+36-16q_3-2q_3^2\right)\Big)\Bigg)
\end{multline*}
Many of the partial evaluations are repeated for different summands of the expression for the Hecke operators.
We take advantage of this phenomenon by caching the results of evaluations of polynomials in $q_3$ and $q_3^{-1}$.
All the operations are performed using interval arithmetic (via the \texttt{ComplexIntervalField} available in Sage).
While this introduces a small overhead, it frees us from having to keep track of precision loss due to arithmetic operations (and evaluations of the complex exponential function).
Sage gives the final approximation of the Hecke eigenvalue in the form
\begin{verbatim}
  1.0555282184708004141101491800000000000000?e27 + 0.?e-13*I
\end{verbatim}
from which we observe that the answer is most likely the integer
\begin{verbatim}
  1055528218470800414110149180
\end{verbatim}
which is indeed $\lambda_{29}(\Upsilon_{20})$.
The question mark in the floating point number indicates that the last decimal may be incorrect due to rounding errors (but all preceding decimals are guaranteed to be correct).

There are certainly many variants of our choices that deserve further scrutiny and may lead to improved performance.
Here are some of the more interesting ones:
\begin{itemize}
  \item For computing the eigenvalue $\lambda_p$, we chose to focus on the initial evaluation point 
    \begin{equation*}
      Z=\begin{pmatrix}y_{11}i&i\\i&(y_{11}+1)i\end{pmatrix},
    \end{equation*}
    where the parameter $y_{11}$ is (at the moment) determined by trial and error.
    The optimal values of $y_{11}$ for $\Upsilon_{20}$ and small $p$ are listed in the second column of Table~\ref{table:ups20}.
    We note that the dependence of this optimal $y_{11}$ on $p$ appears to be linear in $\log(p)$.

    The choice of $Z$ is significant for another reason: the fact that $Z$ is a ``purely imaginary matrix'' gives an extra symmetry that allows to reduce the number of overall computations by almost a factor of $2$.  Note that the timings listed in Table~\ref{table:ups20} do not incorporate this optimization.
  \item Our experiments indicate that computing the value of $\lambda_p$ accurately using the choice of point $Z$ described above requires truncating the $q$-expansions of the Igusa generators at trace up to $2p$.
    It would be very interesting to see if this trace bound can be lowered; even a small improvement in the trace can reduce the computation time significantly.
    We have observed such phenomena in the case of classical modular forms (treated in~\cite{analytic1}).
\end{itemize}

\subsection{Summary of further computations}
We performed similar numerical experiments with the following forms:
\begin{align*}
  \Upsilon_{22}&=
  61E_4^3\chi_{10}
  -30E_4E_6\chi_{12}
  +5E_6^2\chi_{10}
  -80870400\chi_{10}\chi_{12}
  \\
  \Upsilon_{24\mathrm{a}}&=
  -67E_4^3\chi_{12}
  +78E_4^2E_6\chi_{10}
  -274492800E_4\chi_{10}^2
  +25E_6^2\chi_{12}
  +71539200\chi_{12}^2
  \\
  \Upsilon_{24\mathrm{b}}&=
  +70E_4^3\chi_{12}
  -69E_4^2E_6\chi_{10}
  -214341120E_4\chi_{10}^2
  +53E_6^2\chi_{12}
  -137604096\chi_{12}^2
  \\
  \Upsilon_{26\mathrm{a}}&=
  -22E_4^4\chi_{10}
  -3E_4^2E_6\chi_{12}
  +31E_4E_6^2\chi_{10}
  -96609024E_4\chi_{10}\chi_{12}
  -13806720E_6\chi_{10}^2
  \\
  \Upsilon_{26\mathrm{b}}&=
  973E_4^4\chi_{10}
  +390E_4^2E_6\chi_{12}
  -1255E_4E_6^2\chi_{10}
  +3927813120E_4\chi_{10}\chi_{12}
  -4438886400E_6\chi_{10}^2
  \\
\end{align*}
These have in common that they are all ``interesting'' forms (Skoruppa's terminology and notation), not arising as lifts from lower rank groups.
They also all have rational coefficients (and are very likely the only rational ``interesting'' forms in level one).

As we can see in Table~\ref{table:p23}, while the standard method slows down rapidly with the increase in the weight, the numerical method seems unaffected by the weight (in this range).

  \begin{table}[h]
    \centering
    \begin{tabular}{rrrr}
      \toprule
      $f$ & numerical (s) & standard (s)\\
      \midrule
      $\Upsilon_{20}$  & $57$ & $240$ \\
      $\Upsilon_{22}$  & $59$ & $410$ \\
      $\Upsilon_{24\mathrm{a}}$ & $59$ &  $559$ \\
      $\Upsilon_{24\mathrm{b}}$ & $59$ &  $563$ \\
      $\Upsilon_{26\mathrm{a}}$ & $59$ &  $658$ \\
      $\Upsilon_{26\mathrm{b}}$ & $60$ &  $659$ \\
      \bottomrule
    \end{tabular}
    \caption{Benchmarks comparing the numerical and standard algorithms for computing the Hecke eigenvalue $\lambda_{23}$ of the rational ``interesting'' eigenforms. The timings are rounded to the nearest second.}
    \label{table:p23}
  \end{table}

As we increase the weight further, we encounter ``interesting'' eigenforms defined over number fields of increasing degree.
Our implementation treats these in the same way as the rational eigenforms; the algebraic numbers appearing in the expression of an eigenform as a polynomial in the Igusa generators are first embedded into the \texttt{ComplexIntervalField} with the working precision, and the computations are then done exclusively with complex intervals.

We illustrate this with a number of examples from the L-functions and Modular Forms Database (LMFDB~\cite{lmfdb}): $\Upsilon_{28},\Upsilon_{30},\dots,\Upsilon_{56}$, contributed by Nils-Peter Skoruppa.
These are representatives of the unique Galois orbit of ``interesting'' Siegel modular eigenforms of level one and weights given by the indices.
We computed the integer closest to the eigenvalues $\lambda_2,\lambda_3,\dots,\lambda_{11}$ of these forms and verified the results against Sho Takemori's
implementation.\footnote{The LMFDB contains only $\lambda_2,\lambda_3$ and $\lambda_5$ for the forms $\Upsilon_{28},\dots,\Upsilon_{48}$.
We are not aware of the other eigenvalues we computed having been published anywhere.}
The timings for $\lambda_{11}$ appear in Table~\ref{table:p11}.
We note once again that the change in weight has only a very minimal effect on the timings for the numerical approach.
The degree of the number field over which each eigenform is defined varies from $3$ for $\Upsilon_{28}$ to $29$ for $\Upsilon_{56}$.

  \begin{table}[h]
    \centering
    \begin{tabular}{rrrr}
      \toprule
      $f$ & numerical (s) & standard (s) & integer closest to $\lambda_{11}(f)$ \\
      \midrule
      $\Upsilon_{28}$  & $5$ & $42$ &
      {\tiny $-5759681178477373721671849774$}
      \\
      $\Upsilon_{30}$  & $5$ & $55$ &
      {\tiny $255840273811994841300205675092$}
      \\
      $\Upsilon_{32}$  & $5$ & $72$ &
      {\tiny $-62889079837500073468061496815555$}
      \\
      $\Upsilon_{34}$  & $5$ & $99$ &
      {\tiny $439086084572485264922509970244600$}
      \\
      $\Upsilon_{36}$  & $5$ & $145$ &
      {\tiny $-1085248116783567484088793200996441965$}
      \\
      $\Upsilon_{38}$  & $5$ & $171$ &
      {\tiny $99082752899176432104304580529696472526$}
      \\
      $\Upsilon_{40}$  & $6$ & $316$ &
      {\tiny $21639993149436935203941512756710465353890$}
      \\
      $\Upsilon_{42}$  & $6$ & $405$ &
      {\tiny $1326433094276015828828131422320612505802642$}
      \\
      $\Upsilon_{44}$  & $6$ & $697$ &
      {\tiny $-216254834133020533289657866886176910904279874$}
      \\
      $\Upsilon_{46}$  & $6$ & $1156$ &
      {\tiny $3025010356797981861229021682270178023420599162$}
      \\
      $\Upsilon_{48}$  & $6$ & $2147$ &
      {\tiny $3623681259607683701352889863246901251092385443364$}
      \\
      $\Upsilon_{50}$  & $6$ & $3558$ &
      {\tiny $-50111326406849287661448298549933139673192742821477$}
      \\
      $\Upsilon_{52}$  & $6$ & $7701$ &
      {\tiny $-33891727074702812676183940887995219801531644658145401$}
      \\
      $\Upsilon_{54}$  & $6$ & $12205$ &
      {\tiny $-4324363734737815894771410628259133851153783375885366874$}
      \\
      $\Upsilon_{56}$  & $7$ & $19290$ &
      {\tiny $807326143967818876211261524740739769895631903544298785221$}
      \\
      \bottomrule
    \end{tabular}
    \caption{Benchmarks comparing the numerical and standard algorithms for computing the Hecke eigenvalue $\lambda_{11}$ of a representative of the unique Galois orbit of ``interesting'' eigenforms in each of the listed weights. The timings are rounded to the nearest second.}
    \label{table:p11}
  \end{table}

\bibliographystyle{plain}
\bibliography{hecke}

\end{document}